\documentclass[11pt]{amsart}
\usepackage{graphicx}
\usepackage{graphics}
\usepackage{amsmath}
\usepackage{amscd}
\usepackage{latexsym}
\usepackage[all]{xy}
\begin{document}
\textwidth 5.5in
\textheight 8.3in
\evensidemargin .75in
\oddsidemargin.75in

\newtheorem{lem}{Lemma}
\newtheorem{conj}{Conjecture}
\newtheorem{defi}{Definition}
\newtheorem{thm}{Theorem}
\newtheorem{cor}{Corollary}
\newtheorem{lis}{List}
\newtheorem{rmk}{Remark}
\newtheorem{que}{Question}
\newtheorem{prop}{Proposition}
\newcommand{\p}[3]{\Phi_{p,#1}^{#2}(#3)}
\def\Z{\mathbb Z}
\def\R{\mathbb R}
\def\g{\overline{g}}
\def\odots{\reflectbox{\text{$\ddots$}}}
\newcommand{\tg}{\overline{g}}
\def\proof{{\bf Proof. }}
\def\ee{\epsilon_1'}
\def\ef{\epsilon_2'}
\title{Heegaard Floer homology of Matsumoto's manifolds}
\author{Motoo Tange}
\thanks{The author is partially supported by JSPS KAKENHI Grant Number 24840006 and 26800031.}
\address{University of Tsukuba, 1-1-1 Tennodai, Tsukuba, Ibaraki 305-8571 Japan}
\email{tange@math.tsukuba.ac.jp}
\date{\today}
\maketitle
\begin{abstract}
We consider a homology sphere $M_n(K_1,K_2)$ presented by two knots $K_1,K_2$ with linking number 1 and framing $(0,n)$.
We call the manifold {\it Matsumoto's manifold}.
We show that there exists no contractible bound of $M_n(T_{2,3},K_2)$ if $n<2\tau(K_2)$ holds.
We also give a formula of Ozsv\'ath-Szab\'o's $\tau$-invariant as the total sum of the Euler numbers of the reduced filtration. 
We compute the $\delta$-invariants of the twisted Whitehead doubles of torus knots and correction terms of the branched covers of the Whitehead doubles.
By using Owens and Strle's obstruction we show that the $12$-twisted Whitehead double of the $(2,7)$-torus knot and the $20$-twisted Whitehead double of the $(3,7)$-torus knot are not slice 
but the double branched covers bound rational homology 4-balls.
These are the first examples having a gap between sliceness and rational 4-ball bound-ness of the double branched cover.
\end{abstract}
\section{Introduction and computational results.}
\subsection{Matsumoto's manifold and the contractible bound-ness.}
Let $K_1,K_2$ be two knots.
We define to be $M_n(K_1,K_2)$ a homology 3-sphere presented by $K_1$, and $K_2$ with geometrically linking number one and 
framing $(0,n)$ respectively. 
See {\sc Figure}~\ref{matsumoto}.
\begin{figure}[htpb]
\begin{center}\includegraphics{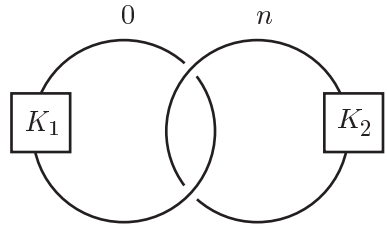}\caption{Matsumoto manifold $M_n(K_1,K_2)$.}\label{matsumoto}\end{center}
\end{figure}
Let $W_n(K_1,K_2)$ be a 4-manifold described by the same picture.
The 4-manifold $W_n(K_1,K_2)$ is a homology $S^2\times S^2-\nu(B^4)$, where $B^4$ is the standard 4-ball.
Y. Matsumoto asked in \cite{K} whether two generators in $H_2(W_0(K_1,K_2))$ can be realized by an embedded $S^2\vee S^2$ (the one-point union of two 2-spheres) or not.
If yes, then $M_0(K_1,K_2)$ bounds a contractible 4-manifold.
If no, then the generators are realized by two embedded Casson handles in $W_0(K_1,K_2)$, and it is an exotic 2-handle homeomorphic to $D^2\times {\Bbb R}^2$, by Freedman's theorem \cite{Fr}.
Following Matsumoto, we call $M_n(K_1,K_2)$ {\it Matsumoto's manifold} in this paper.

If $K_1$ is a slice knot, then $M_n(K_1,K_2)$ bounds a contractible 4-manifold, because 
$M_n(K_1,K_2)$ is the boundary of the slice disk complement $K_1$ together with a 2-handle.
Thus, the attachment has $\pi_1=e$ and $H_2=H_3=0$, hence it is a contractible 4-manifold.
If $K_1$ is not slice, the problem of the contractible bound-ness is unclear.

Let $K$ be a knot in $S^3$.
Then we have
$$M_n(T_{2,3},K)=S^3_1(D_+(K,n)),$$
where $T_{r,s}$ is the positive $(r,s)$-torus knot and
$D_+(K,n)$ is the $n$-twisted (positive-clasped) Whitehead double of $K$.
$S^3_p(K')$ is $p$-surgery of $K'$ in $S^3$.
For example, {\sc Figure}~\ref{Wdouble} is the picture of $D_+(T_{2,3},n)$.
Let $F$ be the figure-8 knot.
Then we have $M_n(F,K)=S^3_{-1}(D_+(K,n))$.
We argue the existence of contractible bounds of $M_n(T_{2,3},K)$ in the present paper.
\begin{figure}[htbp]
\begin{center}
\includegraphics{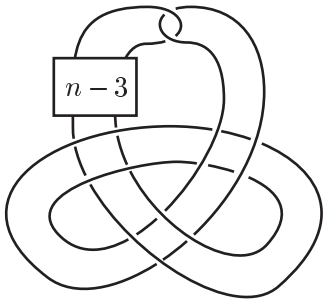}
\caption{The $n$-twisted Whitehead double of the trefoil.}
\label{Wdouble}
\end{center}
\end{figure}

If $K_2$ is a slice knot, then $D_+(K_2,0)$ is slice as in \cite{Kau}.
Thus, $M_0(T_{2,3},K_2)$ has a contractible bound.

The Alexander polynomial of the $n$-twisted Whitehead double is as follows:
$$\Delta_{D_+(K,n)}(t)=-nt+2n+1-nt^{-1}.$$
The result in \cite{FM} says that
if $K'$ is slice, then the Alexander polynomial is of form $\Delta_{K'}(t)=f(t)f(t^{-1})$, where $f(t)$ is a polynomial with integer coefficients.
Hence, if $D_+(K,n)$ is slice, then $\Delta_{D_+(K,n)}$ must be of form $f(t)f(t^{-1})$.
This condition is equivalent to $n=m(m+1)$ for some integer $m$.

Table~\ref{tableofslice} is a list of well-known facts about sliceness of $D_+(K,n)$ and 
contractible bound-ness of $M_n(K_1,K_2)$.
\begin{table}[htbp]
\begin{tabular}{|c|l|l|}\hline
1976&Casson, Rolfsen&$D_+(T_{2,3},6)$ is slice. \cite{R, Kau}\\\hline
1984&Maruyama&$M_6(T_{2,3},T_{2,3})$ bounds a contractible 4-manifold. \cite{M}\\\hline
1997&Akbulut&$M_0(T_{2,3},T_{2,3})$ bounds no contractible 4-manifold. \cite{A}\\\hline
2006&Bar-Natan&$D_+(T_{2,3},2)$ is not slice. \cite{Ba}\\\hline
2007&Hedden&The computation of $\tau(D_+(K,n))$. \cite{H}\\\hline
2012&Collins&$D_+(T_{p,q},n)$ is slice for any relatively prime $(p,q)$\\
&&for at most one $n$. \cite{JC}\\\hline
2013&Tsuchiya&$M_{2n+1}(T_{2,3},T_{2,3})$ bounds no contractible 4-manifold. \cite{T}\\\hline
\end{tabular}
\caption{The well-known results for the sliceness of $D_+(T_{p,q},n)$ and 
the contractible bound-ness of $M_n(T_{2,3},K)$.}
\label{tableofslice}
\end{table}
We here state Ozsv\'ath-Szab\'o's $\tau$-invariant formula of $D_+(K,n)$ by Hedden.
The $\tau$-invariant by Ozsv\'ath and Szab\'o is a homomorphism from the smooth knot concordance group to integers, i.e., $\tau:{\mathcal C}_{\text{sm}}\to {\Bbb Z}$.
\begin{thm}[\cite{H}]
\label{Hedd}
Let $K$ be a knot in $S^3$.
$$\tau(D_+(K,n))=\begin{cases}0&n\ge 2\tau(K)\\1&n<2\tau(K).\end{cases}$$
In particular, if $n<2\tau(K)$, then $D_+(K,n)$ is not slice.
\end{thm}

We can easily compute the Casson invariant by using the Dehn surgery formula as follows:
\begin{equation}
\label{casson}
\lambda(M_n(T_{2,3},K))=\frac{1}{2}\cdot\Delta''_{D_+(K,n)}(t)|_{t=1}=-n.
\end{equation}
One of the main purposes of the present paper is to compute of $HF^+$ of $S^3_1(D_+(T_{2,3},n))$
and to generalize to $S^3_1(D_+(K,n))$ and to discuss the contractible bound-ness of $M_n(T_{2,3},K)$.
Let $M_n(K)$ denote $M_n(T_{2,3},K)$.
\begin{thm}
\label{main}
The Heegaard Floer homology of $M_n(T_{2,3})$ is computed as follows:
$$HF^+(M_n(T_{2,3}))=\begin{cases}{\mathcal T}^+_{(0)}\oplus HF_{\text{\normalfont{red}}}(M_n(T_{2,3}))&n\ge 2\\{\mathcal T}^+_{(-2)}\oplus HF_{\text{\normalfont{red}}}(M_n(T_{2,3}))&n<2,\end{cases}$$
and further
$$HF_{\text{\normalfont red}}(M_n(T_{2,3}))\cong
\begin{cases}
{\Bbb F}_{(-1)}^{n-2}\oplus{\Bbb F}_{(-3)}^2&n\ge 2\\
{\Bbb F}_{(-2)}^{1-n}\oplus {\Bbb F}_{(-3)}^2&n<2.\end{cases}$$
\end{thm}
This computation is generalized to the case of any knot $K$ as well as $T_{2,3}$ as follows:
\begin{thm}
\label{genmain}
Let $K$ be a knot in $S^3$ with genus $g$.
The Heegaard Floer homology of $M_n(K)$ is computed as follows:
$$HF^+(M_n(K))\cong 
\begin{cases}
{\mathcal T}^+_{(0)}\oplus HF_{\text{\normalfont red}}(M_n(K))&n\ge 2\tau(K)\\
{\mathcal T}^+_{(-2)}\oplus HF_{\text{\normalfont red}}(M_n(K))&n<2\tau(K)
\end{cases}$$
and further,
$$HF_{\text{\normalfont red}}(M_n(K))\cong
\begin{cases}
{\Bbb F}_{(-1)}^{n-2\tau(K)}\overset{g}{\underset{i=-g}{\oplus}}H_{\ast+1}(\widetilde{\mathcal F}(K,i))^2&n\ge 2\tau(K)\\
{\Bbb F}_{(-2)}^{2\tau(K)-n-1}\overset{g}{\underset{i=-g}{\oplus}}H_{\ast+1}(\widetilde{\mathcal F}(K,i))^2&n< 2\tau(K).
\end{cases}$$
\end{thm}
The reduced knot filtration $\widetilde{{\mathcal F}}(K,i)$ will be defined in the Section~\ref{genered}.
$\widetilde{\mathcal F}(K,i)$ is a sub-filtration of the knot filtration

This theorem means that from Theorem~\ref{Hedd} we have
$$\tau(D_+(K,n))=-2d(S^3_1(D_+(K,n)))=-2d(M_n(K)).$$

In the same way as the one of the case of $M_n(T_{2,3},K)$ one can also compute the Heegaard Floer homology of $M_n(F,K)=S^3_{-1}(D_+(K,n))$.
Here we state the correction term formula of $M_n(F,K)$.
\begin{thm}
\label{fig8}
Let $K$ be a knot in $S^3$.
Then we have
$$d(M_n(F,K))=0.$$
\end{thm}
From Theorem~\ref{genmain}, the following holds naturally:
\begin{cor}
If $n<2\tau(K)$, then $M_n(K)$ bounds no negative-definite 4-manifold.

In particular, if $n<2\tau(K)$, then $M_n(K)$ bounds no contractible 4-manifold, hence, $D_+(K,n)$
is not slice.
\end{cor}
\begin{proof}
Since in the case of $n<2\tau(K)$, the correction term $d(M_n(K))$ is a negative integer.
The non-negativity of $d(Y^3)$ for a homology 3-sphere $Y^3$ is a necessary condition to have
a negative definite bound for $Y^3$ (see \cite{OS1}).
Hence, in particular, $M_n(K)$ has no contractible bound.
Therefore $D_+(K,n)$ is not slice.
\qed
\end{proof}
The following question is to be solved.
\begin{que}
Let $n$ be an integer with $n\ge 2\tau(K)$.
Does $M_n(K)$ bound a contractible 4-manifold?

For any integer $n$, does $M_n(F,K)$ bound a contractible 4-manifold?
\end{que}
By using the Casson invariant computation (\ref{casson}) of $M_n(K)$, we can give a formula of $\tau(K)$ as the whole sum of Euler numbers of the reduced filtration of $K$.
\begin{cor}
\label{filtEul}
Let $K$ be a knot in $S^3$ with genus $g$.
Then the $\tau$-invariant of $K$ is computed by the following formula:
$$\tau(K)=\sum_{i=-g}^g\chi(\widetilde{{\mathcal F}}(K,i)),$$
where $\widetilde{{\mathcal F}}(K,i)$ is the reduced knot filtration of $K$.
\end{cor}

\subsection{Rational 4-ball bound-ness of $\Sigma_2(D_+(K,n))$ and the $\delta$-invariant.}
\label{ra4ball}
The next purpose of this paper is to consider the rational 4-ball bound-ness of the double branched cover $\Sigma_2(K')$ of a knot $K'$.
This is related to the sliceness of $K'$.
The following always holds:
\begin{center}
$K'$ is slice $\Rightarrow$ $\Sigma_2(K')$ bounds a rational 4-ball.
\end{center}
A rational 4-ball $B$ is a 4-manifold with $H_\ast(B,{\Bbb Q})\cong H_\ast(B^4,{\Bbb Q})$, where $B^4$ is 
the standard 4-ball.
The $\delta$-invariant by Manolescu and Owens is a related invariant to this relationship.
The invariant is defined to be
$$\delta(K')=2d(\Sigma_2(K'),\frak{c}_0),$$
where $\frak{c}_0$ is the canonical Spin$^c$ structure for the unique spin structure on $\Sigma_2(K')$.
The $\delta$-invariant gives a homomorphism
$$\delta:{\mathcal C}_{\text{sm}}\to {\Bbb Z}.$$
Thus, we have the following
\begin{center}$\Sigma_2(K')$ bounds a rational 4-ball$\Rightarrow$ $\delta(K')=0$.\end{center}

They also computed the $\delta$-invariant for the untwisted Whitehead double of any knot or any alternating knot.
\begin{thm}[\cite{MO}]
For any knot $K$ we have $\delta(D_+(K,0))\le 0$ and inequality is strict, if $\tau(K)>0$.
If $K$ is alternating, then $\delta(D_+(K,0))=-4\max\{\tau(K),0\}$.
\end{thm}
Hence, the $\delta$-invariant of the untwisted Whitehead double of $T_{2,2p+1}$ is as follows:
$$\delta(D_+(T_{2,2p+1},0))=-4p.$$

Lisca in \cite{Li} proved that when $K'$ is a 2-bridge knot, conversely if $\Sigma_2(K')$ bounds a rational 4-ball,
then $K'$ is slice.
Namely, she proved the following theorem:
\begin{thm}[\cite{Li}]
Suppose that $K'$ is a 2-bridge knot.
Then we have
\begin{center}
$K'$ is slice $\Leftrightarrow$ $\Sigma_2(K')$ bounds a rational 4-ball.
\end{center}
\end{thm}
How about the case where $K'$ is the $n$-twisted Whitehead double of a knot $K$?
\begin{que}
\label{rationalballsliceconj}
Suppose that $K'$ is the $n$-twisted Whitehead double.
Then, does the following equivalence
\begin{center}$K'$ is slice $\Leftrightarrow$ $\Sigma_2(K')$ bounds a rational 4-ball.\end{center}
hold?
\end{que}
We give a negative answer for Question~\ref{rationalballsliceconj}
\begin{thm}
\label{2712}
Let $K'$ be $D_+(T_{2,7},12)$ or $D_+(T_{3,7},20)$.
Then $K'$ is not slice but $\Sigma_2(K')$ bounds a rational 4-ball.
\end{thm}

We compute the values of $\delta$ in the cases of $K'=D_+(T_{2,2p+1},n)$ and $D_+(T_{3,3p+1},n)$.
\begin{thm}
\label{deltainvariant}
Let $n$ be a non-negative integer and $p$ a positive integer.
Then we have 
$$\delta(D_+(T_{2,2p+1},n))=-4\max\left\{p-\left\lfloor\frac{n}{2}\right\rfloor,0\right\}$$
and we have 
$$\delta(D_+(T_{3,3p+1},n))=-4\max\left\{2p-\left\lfloor\frac{n}{3}\right\rfloor,0\right\}$$
\end{thm}
We define $t_s$, $t_\tau$, $t_\delta$ and $t_{d_1}$ as follows:
$$t_s(K)=\min\{t\in {\Bbb Z}|s(D_+(K,t))=0\}$$
$$t_\tau(K)=\min\{t\in {\Bbb Z}|\tau(D_+(K,t))=0\}$$
$$t_\delta(K)=\min\{t\in {\Bbb Z}|\delta(D_+(K,t))=0\}.$$
$$t_{d_1}(K)=\min\{t\in {\Bbb Z}|d(S^3_1(D_+(K,t)))=0\}.$$
Hedden in \cite{H} showed that $t_\tau(K)=2\tau(K)$ and our result says that $t_{d_1}(K)=2\tau(K)$.
Hedden and Ording conjectured $t_s(T_{2,2p+1})=3p-1$ in \cite{HO}.
Theorem~\ref{deltainvariant} shows that 
$$t_\delta(T_{2,2p+1})=t_\tau(T_{2,2p+1})=2p$$
and 
$$t_\delta(T_{3,3p+1})=t_\tau(T_{3,3p+1})=6p$$

We say $K$ to be a positive L-space knot, if the positive integral Dehn surgery of $K$ is an L-space.
\begin{thm}
\label{lspathm}
If $K$ is a positive torus knot (or in general, any positive L-space knot), then we have
$$t_{\delta}(K)=2\tau(K).$$
\end{thm}
We ask the following question.
\begin{que}
Does there exist a non-L-space knot $K$ with $t_{\delta}(K)\neq 2\tau(K)$?
\end{que}
\begin{thm}
\label{RationalBall}
Suppose that $n$ is a non-negative integer and $s=2,3$.
If $\Sigma_2(D_+(T_{s,sp+1},n))$ bounds a rational 4-ball, then $(s,p,n)=(2,1,6)$, $(2,3,12)$, $(3,1,12)$ or $(3,2,20)$.
\end{thm}

These results are due to Owens and Strle's refinement of the $\delta$-invariant in \cite{Ow}.
Thus we have the following corollary:
\begin{cor}
\label{converse}
Let $n$ be a non-negative integer and $p$ a positive integer.
Then we have:
\begin{center}
$\Sigma_2(D_+(T_{2,2p+1}),n)$ bounds a rational 4-ball $\Leftrightarrow (p,n)=(1,6)$ or $(3,12)$.
\end{center}
\begin{center}
$\Sigma_2(D_+(T_{3,3p+1}),n)$ bounds a rational 4-ball $\Leftrightarrow (p,n)=(1,12)$ or $(2,20)$.
\end{center}
\end{cor}
\begin{proof}
The two cases $(1,6)$ and $(1,12)$ are immediate because $D_+(T_{2,3},6)$ and $D_+(T_{3,4},12)$ are slice (by Casson).
The second two cases are due to Theorem~\ref{2712}.
\qed\end{proof}

We can easily extend the proof (by Casson) of the sliceness of $D_+(T_{2,3},6)$ by Kauffman in \cite{Kau} 
to the sliceness of $D_+(T_{n,n+1},n(n+1))$.
We raise a question of the rational 4-ball bound-ness of the double branched cover of $D_+(T_{p,q},m(m+1))$.
Collins' obstruction in \cite{JC} denies the sliceness of $D_+(T_{p,q},m(m+1))$ unless $\{p,q\}=\{m,m+1\}$.
\begin{que}
Let $n, p$ and $q$ be integers with $n=m(m+1)$ and $\{p,q\}\neq \{m,m+1\}$ for some integer $m$.
When does $\Sigma_2(D_+(T_{p,q},n))$ bound a rational 4-ball?
\end{que}

\section*{Acknowledgements}
Our work in this paper was inspired by Masatsuna Tsuchiya's talk in the handle seminar in 2013 at Tokyo Institute of Technology, Ookayama.
The author is deeply grateful for his talk.
The author thanks Tetsuya Abe in Osaka City University for giving useful comments.

\section{Heegaard Floer homology of $M_n(T_{2,3})$ and $M_n(K)$.}
We prove Theorem~\ref{main}.
Let $M_n$ denote $M_n(T_{2,3})$.

{\bf Proof of Theorem~\ref{main}.}
Ozsv\'ath-Szab\'o's $\tau$-invariant of $T_{2,3}$ is $1$.
From the equality $M_n=S^3_1(D_+(T_{2,3},n))$ and Proposition~7.2 in \cite{H}, we have
$$\widehat{HF}(M_n)=\widehat{HF}(S^3_1(D_+(T_{2,3},n)))$$
$$\cong\begin{cases}{\Bbb F}_{(-1)}^{n-4}\oplus {\Bbb F}_{(0)}^{n-3}\oplus V&n\ge 2\\
{\Bbb F}_{(-1)}^{-1-n}\oplus {\Bbb F}_{(-2)}^{2-n}\oplus {\Bbb F}_{(0)}^{-2}\oplus V&n< 2.\end{cases}$$
The negative exponent means the quotient operation in place of the direct sum operation.
The summand $V$ is isomorphic to 
$$V=\overset{1}{\underset{i=-1}{\oplus}}[H_{\ast}({\mathcal F}(T_{2,3},i))]^2\overset{1}{\underset{i=-1}{\oplus}}[H_{\ast+1}({\mathcal F}(T_{2,3},i))]^2.$$
The chain complex $CFK^\infty(T_{2,3})$ is as in {\sc Figure}~\ref{trefoilcomplex}.
\begin{figure}[htbp]
\centering
\unitlength 0.1in
\begin{picture}( 21.2000, 22.4000)(  4.0000,-28.0000)
%
\special{pn 8}%
\special{sh 1}%
\special{ar 2000 800 10 10 0  6.28318530717959E+0000}%
%
\special{pn 8}%
\special{sh 1}%
\special{ar 1200 2400 10 10 0  6.28318530717959E+0000}%
%
\special{pn 8}%
\special{sh 1}%
\special{ar 1200 2400 10 10 0  6.28318530717959E+0000}%
%
\special{pn 8}%
\special{sh 1}%
\special{ar 1200 2000 10 10 0  6.28318530717959E+0000}%
%
\special{pn 8}%
\special{sh 1}%
\special{ar 1200 1600 10 10 0  6.28318530717959E+0000}%
%
\special{pn 8}%
\special{sh 1}%
\special{ar 2400 800 10 10 0  6.28318530717959E+0000}%
%
\special{pn 8}%
\special{sh 1}%
\special{ar 2400 1200 10 10 0  6.28318530717959E+0000}%
%
\special{pn 8}%
\special{sh 1}%
\special{ar 1600 2000 10 10 0  6.28318530717959E+0000}%
%
\special{pn 8}%
\special{sh 1}%
\special{ar 1600 1600 10 10 0  6.28318530717959E+0000}%
%
\special{pn 8}%
\special{sh 1}%
\special{ar 1600 1200 10 10 0  6.28318530717959E+0000}%
%
\special{pn 8}%
\special{sh 1}%
\special{ar 2000 1600 10 10 0  6.28318530717959E+0000}%
%
\special{pn 8}%
\special{sh 1}%
\special{ar 2000 1200 10 10 0  6.28318530717959E+0000}%
%
\special{pn 8}%
\special{pa 2350 800}%
\special{pa 2050 800}%
\special{fp}%
\special{sh 1}%
\special{pa 2050 800}%
\special{pa 2118 820}%
\special{pa 2104 800}%
\special{pa 2118 780}%
\special{pa 2050 800}%
\special{fp}%
%
\special{pn 8}%
\special{pa 1950 1200}%
\special{pa 1650 1200}%
\special{fp}%
\special{sh 1}%
\special{pa 1650 1200}%
\special{pa 1718 1220}%
\special{pa 1704 1200}%
\special{pa 1718 1180}%
\special{pa 1650 1200}%
\special{fp}%
%
\special{pn 8}%
\special{pa 1550 1600}%
\special{pa 1250 1600}%
\special{fp}%
\special{sh 1}%
\special{pa 1250 1600}%
\special{pa 1318 1620}%
\special{pa 1304 1600}%
\special{pa 1318 1580}%
\special{pa 1250 1600}%
\special{fp}%
%
\special{pn 8}%
\special{pa 2400 850}%
\special{pa 2400 1150}%
\special{fp}%
\special{sh 1}%
\special{pa 2400 1150}%
\special{pa 2420 1084}%
\special{pa 2400 1098}%
\special{pa 2380 1084}%
\special{pa 2400 1150}%
\special{fp}%
%
\special{pn 8}%
\special{pa 2000 1250}%
\special{pa 2000 1550}%
\special{fp}%
\special{sh 1}%
\special{pa 2000 1550}%
\special{pa 2020 1484}%
\special{pa 2000 1498}%
\special{pa 1980 1484}%
\special{pa 2000 1550}%
\special{fp}%
%
\special{pn 8}%
\special{pa 1600 1650}%
\special{pa 1600 1950}%
\special{fp}%
\special{sh 1}%
\special{pa 1600 1950}%
\special{pa 1620 1884}%
\special{pa 1600 1898}%
\special{pa 1580 1884}%
\special{pa 1600 1950}%
\special{fp}%
%
\special{pn 8}%
\special{pa 1200 2050}%
\special{pa 1200 2350}%
\special{fp}%
\special{sh 1}%
\special{pa 1200 2350}%
\special{pa 1220 2284}%
\special{pa 1200 2298}%
\special{pa 1180 2284}%
\special{pa 1200 2350}%
\special{fp}%
%
\special{pn 8}%
\special{sh 1}%
\special{ar 800 2000 10 10 0  6.28318530717959E+0000}%
%
\special{pn 8}%
\special{pa 1150 2000}%
\special{pa 850 2000}%
\special{fp}%
\special{sh 1}%
\special{pa 850 2000}%
\special{pa 918 2020}%
\special{pa 904 2000}%
\special{pa 918 1980}%
\special{pa 850 2000}%
\special{fp}%
%
\special{pn 8}%
\special{pa 400 2400}%
\special{pa 2400 2400}%
\special{fp}%
\special{sh 1}%
\special{pa 2400 2400}%
\special{pa 2334 2380}%
\special{pa 2348 2400}%
\special{pa 2334 2420}%
\special{pa 2400 2400}%
\special{fp}%
%
\special{pn 8}%
\special{pa 800 2800}%
\special{pa 800 600}%
\special{fp}%
\special{sh 1}%
\special{pa 800 600}%
\special{pa 780 668}%
\special{pa 800 654}%
\special{pa 820 668}%
\special{pa 800 600}%
\special{fp}%
\put(25.2000,-24.4000){\makebox(0,0)[lb]{$i$}}%
\put(8.9000,-7.3000){\makebox(0,0)[lb]{$j$}}%
\end{picture}%
\\
\caption{$CFK^\infty(T_{2,3})$.}
\label{trefoilcomplex}
\end{figure}
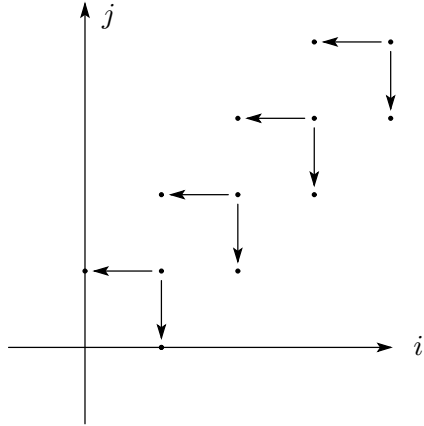
Therefore, we have
\begin{equation}
\label{hat}
H_{\ast}({\mathcal F}(T_{2,3},i))\cong\begin{cases}{\Bbb F}_{(-2)}&i=-1\\{\Bbb F}_{(0)}&i=1\\0&otherwise.\end{cases}
\end{equation}
As a result, the $V$-part is computed by 
$$V={\Bbb F}_{(-2)}^2\oplus {\Bbb F}_{(0)}^2\oplus {\Bbb F}_{(-3)}^2\oplus {\Bbb F}_{(-1)}^2.$$
Therefore, we have
\begin{equation}
\label{alt}
\widehat{HF}(M_n)=
\begin{cases}
{\Bbb F}_{(0)}^{n-1}\oplus{\Bbb F}_{(-1)}^{n-2}\oplus{\Bbb F}_{(-2)}^2\oplus{\Bbb F}_{(-3)}^2&n\ge 2\\
{\Bbb F}_{(-1)}^{1-n}\oplus{\Bbb F}_{(-2)}^{4-n}\oplus{\Bbb F}_{(-3)}^2&n< 2.\end{cases}
\end{equation}
Here we use the exact triangle
$$\cdots\to \widehat{HF}(M_n)\overset{j_\ast}{\to} HF^+(M_n(T_{2,3}))\overset{\times U}{\to} HF^+(M_n)\overset{\delta}{\to} \widehat{HF}(M_n)\to \cdots.$$
The map $j_\ast$ is the one induced by the natural injection.
The multiplication of $U$ is Ozsv\'ath-Szab\'o's usual action lowering the degree by two.
The connecting homomorphism $\delta$ shifts the degree by $1$.

We compute the correction term $d(M_n)$.
$$HF^+(S^3_1(D_+(T_{2,3},n)))\cong H_{\ast}(C\{\max(i,j)\ge 0\}),$$
where the chain complex $C$ is $CFK^\infty(D_+(T_{2,3},n))$.

{\bf The case of $n\ge 2$.}
$C\{i=0\}$ is filtered chain homotopic to 
$$C\{(0,j)\}\simeq\begin{cases}
{\Bbb F}_{(1)}^{n-4}\overset{1}{\underset{i=-1}{\oplus}}H_{\ast-1}({\mathcal F}(T_{2,3},i))^2&j=1\\
{\Bbb F}_{(0)}^{2n-7}\overset{1}{\underset{i=-1}{\oplus}}H_{\ast}({\mathcal F}(T_{2,3},i))^4&j=0\\
{\Bbb F}_{(-1)}^{n-4}\overset{1}{\underset{i=-1}{\oplus}}H_{\ast+1}({\mathcal F}(T_{2,3},i))^2&j=-1,
\end{cases}$$
due to the result in \cite{H}.
The boundary maps:
$$\partial_1^k: C\{(0,k)\}\to C\{(0,k-1)\},$$
on $CFK^\infty(D_+(T_{2,3},n))$ consists of
following the notation in \cite{H}.
We denote the map:
$$C\{(k,0)\}\to C\{(k-1,0)\}$$
obtained by exchanging $i$ and $j$ by $\delta_1^k$.
The boundary map $\partial^0_1$ induces a surjective map on the group $\widehat{HFK}$ (Lemma~6.1 in \cite{H})

The component in $\widehat{HF}(M_n)$ attaining the minimal degree in the non-torsion part in $HF^+(M_n)$ is
located at $(i,j)=(0,0)$, that is, $x\in {\Bbb F}_{(0)}\subset C\{(0,0)\}$.
This generator $x$ vanishes by the boundary map $\partial^0_1:C\{(0,0)\}\to C\{(0,-1)\}$,
because it is the generator of $\widehat{HF}(S^3)$.
It also vanishes by the map $\delta^0_1:C\{(0,0)\}\to C\{(-1,0)\}$.
Hence $x$ is a generator in $HF^+(M_n)$ and it becomes clearly $U\cdot x=0$ and $\text{gr}(x)=0$,
where $\text{gr}$ stands for the absolute grading on $C$.
This means $d(M_n)=0$.

{\bf The case of $n< 2$.}
$C\{i=0\}$ is filtered chain homotopic to 
$$C\{(0,j)\}\simeq
\begin{cases}
{\Bbb F}_{(0)}^{2-n}\oplus\left(\overset{1}{\underset{i=-1}{\oplus}} H_{\ast-1}({\mathcal F}(T_{2,3},i))^2\right)/{\Bbb F}_{(1)}^2&j=1\\
{\Bbb F}_{(-1)}^{3-2n}\oplus\left(\overset{1}{\underset{i=-1}{\oplus}}H_{\ast}({\mathcal F}(T_{2,3},i))^4\right)/{\Bbb F}_{(0)}^4&j=0\\
{\Bbb F}_{(-2)}^{2-n}\oplus\left(\overset{1}{\underset{i=-1}{\oplus}}H_{\ast+1}({\mathcal F}(T_{2,3},i))^2\right)/{\Bbb F}_{(-1)}^2&j=-1
\end{cases}$$
$$\cong
\begin{cases}
{\Bbb F}_{(0)}^{2-n}\oplus H_{\ast-1}({\mathcal F}(T_{2,3},-1))^2&j=1\\
{\Bbb F}_{(-1)}^{3-2n}\oplus H_{\ast}({\mathcal F}(T_{2,3},-1))^4&j=0\\
{\Bbb F}_{(-2)}^{2-n}\oplus H_{\ast+1}({\mathcal F}(T_{2,3},-1))^2&j=-1,
\end{cases}$$
due to the result in \cite{H}.
The generator $x$ in $\widehat{HF}(S^3)$ lies in $C\{(0,1)\}$.
That is, $x\in {\Bbb F}_{(0)}\subset C\{(0,1)\}$.
The boundary map
$$\delta^0_1:C\{(0,0)\}\to C\{(-1,0)\}$$
is surjective due to \cite{H}.
The $U$-action $C\{(-1,0)\}\ni U\cdot x\neq 0$, namely, $U\cdot x$ is the minimal generator in $HF^+(M_n)$.
Thus we have $d(M_n)=\text{gr}(U\cdot x)=-2$.

We put $HF^+(M_n)\cong {\mathcal T}^+_{(d(M_n))}\overset{m}{\underset{i=1}{\oplus}} W_i$, where $W_i={\Bbb F}[n_i]_{(d_i)}\cong {\Bbb F}[U]/U^{n_i}$
with minimal degree $d_i$.
Then we have
$$\widehat{HF}(M_n)={\Bbb F}_{(d(M_n))}\overset{m}{\underset{i=1}{\oplus}}({\Bbb F}_{(d_i)}\oplus{\Bbb F}_{(d_i+2n_i-1)}).$$
From the computation of $\widehat{HF}(M_n)$, the number of the components is
$$m=\begin{cases}n&n\ge 2\\3-n&n<2.\end{cases}$$
If some integer $i$ has $n_i>1$, then there exists in the $\widehat{HF}(M_n)$ a pair of two summands with the difference of $2n_i-1\ge 3$.
The pair is just the case $n_i=2$ and the only pair is ${\Bbb F}_{(0)}$ and ${\Bbb F}_{(-3)}$ in the case of $n\ge 2$,
due to (\ref{alt}).
Hence, the number of the pairs is at most two.
\begin{prop}
There exists no such a pair.
Therefore, in the case of $n\ge 2$, we have
$$HF^+(M_n)={\mathcal T}^+_{(0)}\oplus {\Bbb F}_{(-1)}^{n-2}\oplus{\Bbb F}_{(-3)}^2$$
and in the case of $n<2$, we have
$$HF^+(M_n)={\mathcal T}^+_{(-2)}\oplus {\Bbb F}_{(-2)}^{1-n}\oplus{\Bbb F}_{(-3)}^2.$$
\end{prop}
\begin{proof}
We may consider the $n\ge 2$ case.
Suppose that there exist such two pairs in $\widehat{HF}(M_n)$.
Then the components are ${\Bbb F}[2]_{(-3)}^2$ and the remaining part is 
$${\mathcal T}^+_{(0)}\oplus{\Bbb F}_{(-1)}^{n-4}\oplus{\Bbb F}_{(-2)}^2.$$
The Casson invariant becomes $\lambda(M_n)=-4-(n-4)+2=-n+2$.
This is contradiction about (\ref{casson}).

Suppose that there exists such a single pair ${\Bbb F}_{(0)}$, ${\Bbb F}_{(-3)}$ in $\widehat{HF}(M_n)$.
Then the component is ${\Bbb F}[2]_{(-3)}$ and the remaining part is 
$${\mathcal T}^+_{(0)}\oplus{\Bbb F}_{(-1)}^{n-3}\oplus{\Bbb F}_{(-2)}\oplus{\Bbb F}_{(-3)}.$$
The Casson invariant is $\lambda=-2-(n-3)+1-1=-n+1$.
This is contradiction about (\ref{casson}).
Therefore since we have $n_i=1$ for any $i$, we get the required computation of $HF^+(M_n)$.
\qed
\end{proof}
This proposition means Theorem~\ref{main}.
In the case of $n=6$, we can also check our computation (Theorem~\ref{main}) by N\'emethi's algorithm (\cite{N}) on any plumbed 3-manifold with at most one bad vertex.
In fact for $M_6$ we can construct the negative definite bound as in Figure~\ref{nega}.
The multiplicity $-1$ vertex is the only bad vertex.
Then $HF^+(-M_6)$ can be computed as follows:
$$HF^+(-M_6)={\mathcal T}^+_{(0)}\oplus {\Bbb F}_{(0)}^4\oplus {\Bbb F}_{(2)}^2.$$
For example, use HFNem by MAGMA code in \cite{Ka}.
By reversing the orientation, we get
$$HF^+(M_6)={\mathcal T}^+_{(0)}\oplus {\Bbb F}_{(-1)}^4\oplus {\Bbb F}_{(-3)}^2.$$
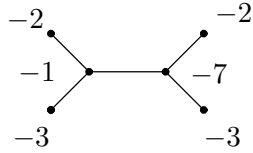
\begin{figure}[htbp]
\centering
\unitlength 0.1in
\begin{picture}( 10.9000,  6.5000)(  5.7000,-10.3000)
%
\special{pn 8}%
\special{pa 800 600}%
\special{pa 1000 800}%
\special{fp}%
\special{pa 1000 800}%
\special{pa 800 1000}%
\special{fp}%
\special{pa 1000 800}%
\special{pa 1400 800}%
\special{fp}%
\special{pa 1400 800}%
\special{pa 1600 600}%
\special{fp}%
\special{pa 1400 800}%
\special{pa 1600 1000}%
\special{fp}%
%
\special{pn 20}%
\special{sh 1}%
\special{ar 1600 600 10 10 0  6.28318530717959E+0000}%
\special{sh 1}%
\special{ar 1600 1000 10 10 0  6.28318530717959E+0000}%
\special{sh 1}%
\special{ar 1400 800 10 10 0  6.28318530717959E+0000}%
\special{sh 1}%
\special{ar 1000 800 10 10 0  6.28318530717959E+0000}%
\special{sh 1}%
\special{ar 800 600 10 10 0  6.28318530717959E+0000}%
\special{sh 1}%
\special{ar 800 1000 10 10 0  6.28318530717959E+0000}%
\special{sh 1}%
\special{ar 800 1000 10 10 0  6.28318530717959E+0000}%
\put(5.7000,-5.8000){\makebox(0,0)[lb]{$-2$}}%
\put(6.0000,-12.0000){\makebox(0,0)[lb]{$-3$}}%
\put(16.0000,-12.0000){\makebox(0,0)[lb]{$-3$}}%
\put(16.6000,-5.5000){\makebox(0,0)[lb]{$-2$}}%
\put(15.3000,-8.7000){\makebox(0,0)[lb]{$-7$}}%
\put(6.3000,-8.6000){\makebox(0,0)[lb]{$-1$}}%
\end{picture}%
\\
\caption{The negative definite plumbing of $M_6$.}
\label{nega}
\end{figure}
\subsection{The Heegaard Floer homology of $M_n(K)$.}
We define $\widehat{CF}(K)={\bigcup}_{i\in{\Bbb Z}}{\mathcal F}(K,i)$, i.e., it is chain homotopy equivalent to $\widehat{CF}(K)\simeq \widehat{CF}(S^3)$.
Here we define to be $\epsilon_i$ the composition
$$\epsilon_i:{\mathcal F}(K,i)\hookrightarrow \widehat{CF}(K)\to {\Bbb F}_{(0)}$$ for any $i$,
where the last map is the one that 
the homological generator is mapped to $1$ and any other elements to $0$.
Furthermore, since the map $\widehat{CF}(K)\to {\Bbb F}_{(0)}$ is splittable, we obtain the natural decomposition $\widehat{CF}(K)\cong \widetilde{CF}(K)\oplus {\Bbb F}$.
We denote the kernel of $\varphi$ by
$$\widetilde{\mathcal F}(K,i):=\ker(\epsilon_i).$$
Then $\widetilde{\mathcal F}(K,i)$ is a filter on $\widetilde{CF}(K):=\bigcup_{i\in{\Bbb Z}}\widetilde{\mathcal F}(K,i)$.
The chain complex $\widetilde{CF}(K)$ is acyclic, because $\widehat{CF}(K)\to {\Bbb F}_{(0)}$ induces an isomorphism on the homology.
We say $\widetilde{\mathcal F}(K,i)$ to be a {\it reduced knot filtration}.

{\bf Proof of Theorem~\ref{genmain}.}
In the same way as the case where $K$ is the right-handed trefoil,
$$d(M_n(K))=\begin{cases}0&n\ge 2\tau(K)\\-2&n<2\tau(K).\end{cases}$$
$$HF_{\text{\normalfont red}}(M_n(K))\cong\begin{cases}
{\Bbb F}_{(-1)}^{n-2g-2}\overset{g}{\underset{i=-g}{\oplus}}H_{\ast+1}({\mathcal F}(K,i))^2&n\ge 2\tau(K)\\
{\Bbb F}_{(-2)}^{2\tau(K)-n-1}\oplus{\Bbb F}_{(-1)}^{2\tau(K)-2g-2}\overset{g}{\underset{i=-g}{\oplus}}H_{\ast+1}({\mathcal F}(K,i))^2&n< 2\tau(K)
\end{cases}
$$

For each integer $i$ with $i\ge \tau(K)$, $\epsilon_i$ is a non-trivial map and $H_{\ast+1}({\mathcal F}(K,i))$ contains at least one summand ${\Bbb F}_{(-1)}$,
which is non-trivial via the map $\epsilon_i$.
For each integer $i$ with $i<\tau(K)$, $\epsilon_i$ is the 0-map.
We thus have $\widetilde{\mathcal F}(K,i)\cong {\mathcal F}(K,i)$ in such an integer $i$.
Thus, we have the following:
\begin{eqnarray*}
&&{\Bbb F}_{(-1)}^{2\tau(K)-2g-2}\overset{g}{\underset{i=-g}{\oplus}}H_{\ast+1}({\mathcal F}(K,i))^2\\
&\cong&\overset{\tau(K)-1}{\underset{i=-g}{\oplus}}H_{\ast+1}({\mathcal F}(K,i))^2\overset{g}{\underset{i=\tau(K)}{\oplus}}\left(H_{\ast+1}({\mathcal F}(K,i))/{\Bbb F}_{(-1)}\right)^2\\
&\cong&\overset{\tau(K)-1}{\underset{i=-g}{\oplus}}H_{\ast+1}(\widetilde{\mathcal F}(K,i))^2\overset{g}{\underset{i=\tau(K)}{\oplus}}H_{\ast+1}(\widetilde{\mathcal F}(K,i))^2\\
&\cong&\overset{g}{\underset{i=-g}{\oplus}}H_{\ast+1}(\widetilde{\mathcal F}(K,i))^2.
\end{eqnarray*}
\hfill\qed

We prove Theorem~\ref{fig8}.\\
{\bf Proof of Theorem~\ref{fig8}.}
The Heegaard Floer homology of $-1$-surgery of a knot $K'$ is computed by using the quotient complex:
$$CFK^\infty(K')\{i\ge 0\text{ and }j\ge 0\},$$
Since the double chain complex $CFK^\infty(D_+(K,n))$ can be seen in the proof of Theorem~\ref{genmain}.

Let $K'=D_+(K,n)$ and $C:=CFK^\infty(K')$.
Suppose that $n\ge 2\tau(K)$.
Let $x_0\in C\{i=0\}$ be an element with non-zero $[x_0]\in H_\ast(C\{i\ge 0\})$.
This element $x_0$ lies in $C\{(0,0)\}$ and clearly survives in $x_0\in C\{i\ge0 \text{ and }j\ge 0\}$ as a non-zero element in ${\Bbb F}_{(0)}\subset H_\ast(C\{(0,0)\})$.
This is the bottom class in the ${\mathcal T}^+$-component in $H_\ast(C\{i\ge 0\text{ and }j\ge 0\})$.
Hence, we have $d(M_n(F,K))=0$.

Suppose that $n<2\tau(K)$.
The same element in $C\{i\ge 0\}$ satisfying the same condition as above is homologous to an element $x_0$ in $C\{(0,1)\}$.
The element is the bottom class in the ${\mathcal T}^+$-component in $H_\ast(C\{i\ge 0\text{ and }j\ge 0\})$ with absolute grading $0$.
Thus $d(M_n(F,K))=0$.
\hfill\qed

\subsection{The total sum of Euler numbers of the reduced knot filtration.}
\label{genered}
As a corollary of the Casson invariant formula (\ref{casson}) and Heegaard Floer homology formula (Theorem~\ref{genmain}) we give a new formula of the $\tau$-invariant:
\begin{cor}
\label{filtEul}
The total Euler number of the reduced knot filtration is $\tau(K)$, namely we have:
$$\tau(K)=\sum_{i=-g}^g\chi(\widetilde{\mathcal F}(K,i))=\sum_{i\in {\Bbb Z}}\chi(H_\ast(\widetilde{\mathcal F}(K,i))).$$
\end{cor}
\begin{proof}
If $n\ge 2\tau(K)$, then we have
$$\lambda(M_n(K))=-n=-(n-2\tau(K))+\sum_{i=-g}^g\chi(H_{\ast+1}(\widetilde{\mathcal F}(K,i))^2).$$
By using this equality, the sum of Euler numbers of the chain complex $\widetilde{\mathcal F}(K,i)$ is $\tau(K)$.
In the case of $n<2\tau(K)$, by the same argument, we get the same result.

Since $\widetilde{{\mathcal F}}(K,i)$ is acyclic in the case of $|i|\gg 0$, the last equality holds.
\qed\end{proof}

\section{The rational 4-ball bound-ness of $\Sigma_2(D_+(T_{2,2p+1},n))$ and $\Sigma_2(D_+(T_{3,3p+1},n))$.}
If a rational homology 3-sphere $Y$ is the boundary of a rational 4-ball,
then $Y$ is said to be that $Y$ {\it bounds a rational 4-ball}.
As mentioned in Section~\ref{ra4ball}, the sliceness of $K'=D_+(K,n)$ is a sufficient condition for $\Sigma_2(K')$
to bound a rational 4-ball.
We show that the sliceness is {\it not} a necessary condition.

First, we compute the $\delta$-invariant (smooth knot concordance invariant) by Manolescu-Owens.
To prove Theorem~\ref{deltainvariant}, we use the following fact:
\begin{prop}
\label{4n+12KK}
For integer $n$ we have
$$\Sigma_2(D_+(K,n))=S^3_{\frac{4n+1}2}(K\#K^r)$$
where $K^r$ is the knot $K$ with the reverse orientation.

If $\Sigma_2(D_+(K,n))$ bounds a rational 4-ball, then $n=m(m+1)$ holds for some integer $m$.
\end{prop}
We notice that the condition $n=m(m+1)$ is also a necessary condition for $S^3_{\frac{4n+1}{2}}(K\#K^r)$ to bound
a rational 4-ball.
For example, one can see the condition in \cite{CG}.

\begin{proof}
The former assertion is an application of the Montesinos trick.
By using the result by Casson and Gordon \cite{CG}, the order of the first homology group of the
double branched cover is $4n+1$ and it is an odd square number.
Thus we have $n=m(m+1)$ for some integer $m$.
\qed
\end{proof}


Before proving Theorem~\ref{deltainvariant},
we introduce the correction term formula for rational Dehn surgery.
Here we give a brief review of the invariants $V_k$ and $H_k$ in \cite{ZW}.

Let $C$ be a double chain complex $C:=CFK^\infty(K)$ for a knot $K$.
Let $A_k^+$ denote
$$C\{i\ge 0\text{ or }j\ge k\}$$
and $B^+$ denote
$$C\{i\ge 0\}.$$
The maps $\frak{v}_k:A_k^+\to B^+$ are defined to be the natural projections (the forgetting map of the Alexander grading	)
$$\frak{v}_k^+:A_k^+\to B^+$$
and the maps $\frak{h}_k:A_k^+\to B^+$ are defined to be the compositions of the horizontal natural projections and the identification
$$\frak{h}_k:A_k^+\to C\{j\ge k\}\to B^+.$$
$\frak{A}_k^T\subset H_\ast(A_k^+)$ is the sub-${\Bbb F}[U]$-module isomorphic to $U^nH_\ast(A_k^+)$ for $n\gg 0$.
The homomorphisms $\frak{v}_k^T,\frak{h}_k^T$ are the restriction maps on $\frak{A}_k^T$
$$\frak{v}_k^T,\frak{h}_k^T:\frak{A}^T\cong{\mathcal T}^+\to H_{\ast}(B^+)\cong {\mathcal T}^+.$$
The maps are equivalent to the multiplication by $U^m$ for some $m\ge 0$.
We define the exponent $m$ to be $V_k$ or $H_k$ respectively.

The correction term formula by Ni and Wu in \cite{ZW} is the following:
$$d(S^3_{p/q}(K),i)=d(L(p,q),i)-2\max\{V_{\lfloor\frac{i}{q}\rfloor},H_{\lfloor\frac{i-p}{q}\rfloor}\}.$$

If $q$ is an even integer, then the canonical Spin$^c$ structure of $S^3_{p/q}(K)$ has $i_0=\frac{p+q-1}{2}$ (see \cite{MT}).
Then, since $V_{\lfloor\frac{i_0}{q}\rfloor}=H_{\lfloor\frac{i_0-p}{q}\rfloor}$, we have
$$d(S^3_{p/q}(K),i_0)=d(L(p,q),i_0)-2V_{\lfloor\frac{p+q-1}{2q}\rfloor}.$$

{\bf Proof of Theorem~\ref{deltainvariant}.}
Suppose that $s=2$ or $3$.
Let $C_{s,p}$ denote the minimal generating complex that $C_{s,p}[U,U^{-1}]$ is a chain complex of $CFK^\infty(T_{s,sp+1})$.
The chain complex $C_{2,1}$ is presented by the left of {\sc Figure}~\ref{ccomplex}.
In general, the differentials of $C_{2,p}$ are stair-shape (see the left in {\sc Figure}~\ref{Cp2gene}).
We define $C_{2,p}^{(2)}$ to be the chain complex which is isomorphic to $C_{2,p}\otimes C_{2,p}$ and that the $(i,j)$-coordinate of the element $x_0$ with the $i$-coordinate minimal is $(-p,p)$
as in {\sc Figure}~\ref{ccomplex}.
Then the knot Floer chain complex of $T_{s,sp+1}\#T_{s,sp+1}^r$ is as follows:
$$CFK^\infty:=CFK^\infty(T_{s,sp+1}\#T_{s,sp+1}^r)=C_{s,p}^{(2)}[U,U^{-1}].$$
To compute the values of $V_{s,p,k}$ for $T_{s,sp+1}\#T_{s,sp+1}^r$ we investigate $C_{s,p}^{(2)}$.
{\sc Figure}~\ref{ccomplex} presents $C_{2,1}$ and $C_{2,1}^{(2)}$.
The number written nearby each of lattice points is the dimension of the vector space corresponding to the point.
Each of elements surrounded by circles is the unique generator in $H_\ast(C_{s,p}^{(2)})\cong {\Bbb F}_{(-2p)}$
namely all the elements are homologous to each other.
The generator represents a non-trivial element in $HF^\infty(S^3)\cong {\mathcal T}^\infty\cong {\Bbb F}[U,U^{-1}]\cdot [x_0]$.

Here we define $A_{s,p,k}^+$ and $B^+$ as follows:
$$C_{s,p}^{(2)}\{i\ge 0\text{ or }j\ge k\}=:A_{s,p,k}^+$$
$$C_{s,p}^{(2)}\{i\ge 0\}=:B^+.$$

The chain complexes $C_{2,p}$ and $C_{2,p}^{(2)}$ are {\sc Figure}~\ref{Cp2gene}.
Each of the elements surrounded by the circles represents a unique homological generator in $H_\ast(C_{2,p}^{(2)})$.
\begin{figure}[htbp]\begin{center}\input{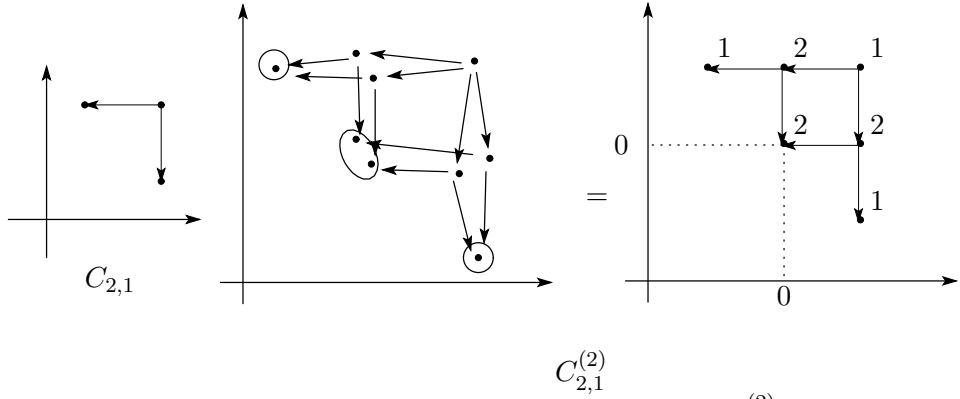}\caption{The chain complexes $C_{2,1}$ and $C_{2,1}^{(2)}$.}\label{ccomplex}\end{center}\end{figure}

Let $F_k=CFK^\infty\{i< 0\text{ and }j< k\}$.
The chain complex $A_{s,p,k}^+$ means $CFK^\infty/F_k$.
Here we define a map $\phi_{i,k}$ to be
$$\phi_{i,k}:C_{2,p}^{(2)}\overset{\varphi_i}{\hookrightarrow} CFK^\infty \to A_{2,p,k}^+.$$
The map $\varphi_i$ is defined to be $\varphi_i(x)=U^{-i}x$ for any element $x\in C_{2,p}^{(2)}$.
The induced map $(\phi_{i,k})_\ast:H_\ast(C_{2,p}^{(2)})={\Bbb F}\to H_\ast(A_{2,p,k}^+)$ is as follows:

\begin{enumerate}
\item The case of $i\ge p$ (the left in {\sc Figure}~\ref{ccomplexxx}).\\
Then $(\phi_{i,k})_\ast$ is injective.
\item The case of $|i|<p$ (the center in {\sc Figure}~\ref{ccomplexxx}).
\begin{itemize}
\item If $2i+1\le k$ then $(\phi_{i,k})_\ast$ is injective.
\item If $2i+1> k$ then $(\phi_{i,k})_\ast$ is the $0$-map.
\end{itemize}
\item The case of $i\le -p$ (the right in {\sc Figure}~\ref{ccomplexxx}).
\begin{itemize}
\item If $k\le i-p$ then $(\phi_{i,k})_\ast$ is injective.
\item If $k> i-p$ then $(\phi_{i,k})_\ast$ is the $0$-map.
\end{itemize}
\end{enumerate}
Indeed, $(\phi_{i,k})_\ast$ is the $0$-map, if and only if $F_k\cap \{(i-\ell,i+\ell)\in {\Bbb Z}^2||\ell|\le p\}\neq \emptyset$.

The values $V_{2,p,k}$ can be interpreted as follows:
$$V_{2,p,k}=\max\{p-i|(\phi_{i,k})_\ast\text{ is injective}\}.$$
\begin{figure}[htbp]\begin{center}\input{kK2.tex}\caption{The images which represent the non-trivial element in the homology of $\varphi_i(C_{2,p}^{(2)})$. Examples of $F_k$: (2) the case of $k\le 2i+1$. (3) the case of $i-p<k$.}\label{ccomplexxx}\end{center}\end{figure}
Hence, we can compute values $V_{2,p,k},H_{2,p,k}$ to become the table below.
\begin{figure}[htbp]
\begin{center}
\input{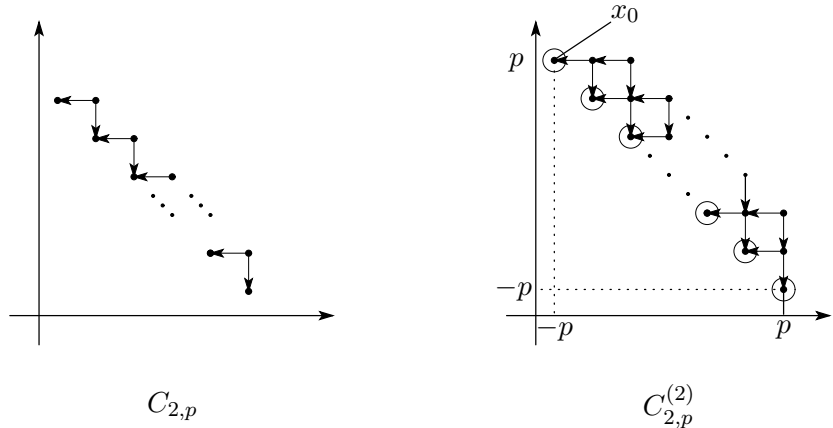}
\caption{The chain complexes $C_{2,p}$, $C_{2,p}^{(2)}$ and an element $x_0\in C_{2,p}^{(2)}\{(-p,p)\}$ which is non-vanishing in $H_\ast(C_{2,p}^{(2)})$.}
\label{Cp2gene}
\end{center}\end{figure}
$$\begin{array}{|c|c|c|c|c|c|c|c|c|}\hline
k&0&1&2&3&\cdots&2p-2&2p-1&k\ge 2p\\\hline
V_{2,p,k}&p&p&p-1&p-1&&1&1&0\\\hline
H_{2,p,k}&p&p+1&p+1&p+2&\cdots&2p-1&2p&k\\\hline\end{array}$$
$$\begin{array}{|c|c|c|c|c|c|c|c|c|c|}\hline
k&k\le -2p&-2p+1&\cdots&-2&-1&0\\\hline
V_{2,p,k}&|k|&2p&\cdots&p+1&p+1&p\\\hline
H_{2,p,k}&0&1&\cdots&p-1&p&p\\\hline\end{array}$$
Thus, we obtain the following values:
$$V_{2,p,k}=\begin{cases}p-\lfloor\frac{k}{2}\rfloor&|k|<2p\\0&k\ge 2p\\-k&k\le -2p.\end{cases}$$
$$H_{s,p,k}=V_{s,p,-k}$$

We move on to the case of $s=3$.
The pictures in {\sc Figure}~\ref{C32} are examples of the chain complexes $C_{3,p}$ and $C_{3,p}^{(2)}$.
The construction of the chain complex $C_{3,p}$ is due to \cite{OSl}.
The differentials of $C_{3,p}$ is the stair-shape whose steps are $p$-steps with slope $(1,-2)$ and consecutively $p$-steps with slope $(2,-1)$ in the positive direction of $i$.
The subcomplex $C_{3,p}^{(2)}\subset CFK^\infty(T_{3,p}\#T_{3,p}^r)$ is the chain complex which satisfies $C_{3,p}^{(2)}\cong C_{3,p}\otimes C_{3,p}$ and 
whose left most element $x_0$ as in {\sc Figure}~\ref{C32} has the coordinate $(-2p,4p)$.

The elements surrounded by circles represent the one which makes the unique generator in $H_\ast(C_{3,p})$ or $H_\ast(C_{3,p}^{(2)})$.
Notice that any torus knot is a lens space knot.
In the same way as the one of the $s=2$ case we obtain the values $V_{3,p,k}$ as follows:
$$V_{3,p,k}=\begin{cases}2p-\lfloor\frac{k}{3}\rfloor&k<6p\\2p-\lfloor\frac{2k}{3}\rfloor&-6p<k<0\\0&k\ge 6p\\-k&k\le -6p.\end{cases}$$

By using the correction term formula of lens spaces in \cite{MT}, we have 
\begin{equation}d(L(2r+1,2),j)=\begin{cases}\frac{(2k-r-2)^2}{2(2r+1)}&j=2k-1\\\frac{4k^2-4kr-4k+r^2}{2(2r+1)}&j=2k.\end{cases}\label{L2formula}
\end{equation}
In the case of $r=2n$, we have $d(L(4n+1,2),i_0)=0$.
Thus, we have 
\begin{eqnarray*}
\delta(D_+(T_{2,2p+1},n))&=&-4V_{2,p,n}=\begin{cases}0&n\ge 2p\\-4(p-\lfloor\frac{n}{2}\rfloor)&0\le n<2p\end{cases}\\
&=&-4\max\left\{p-\left\lfloor\frac{n}{2}\right\rfloor,0\right\}\end{eqnarray*}
and
\begin{eqnarray*}
\delta(D_+(T_{3,3p+1},n))&=&-4V_{3,p,n}=\begin{cases}0&n\ge 6p\\-4(2p-\lfloor\frac{n}{3}\rfloor)&0\le n<6p\end{cases}\\
&=&-4\max\left\{2p-\left\lfloor\frac{n}{3}\right\rfloor,0\right\}.\end{eqnarray*}
\begin{figure}[htbp]\begin{center}\input{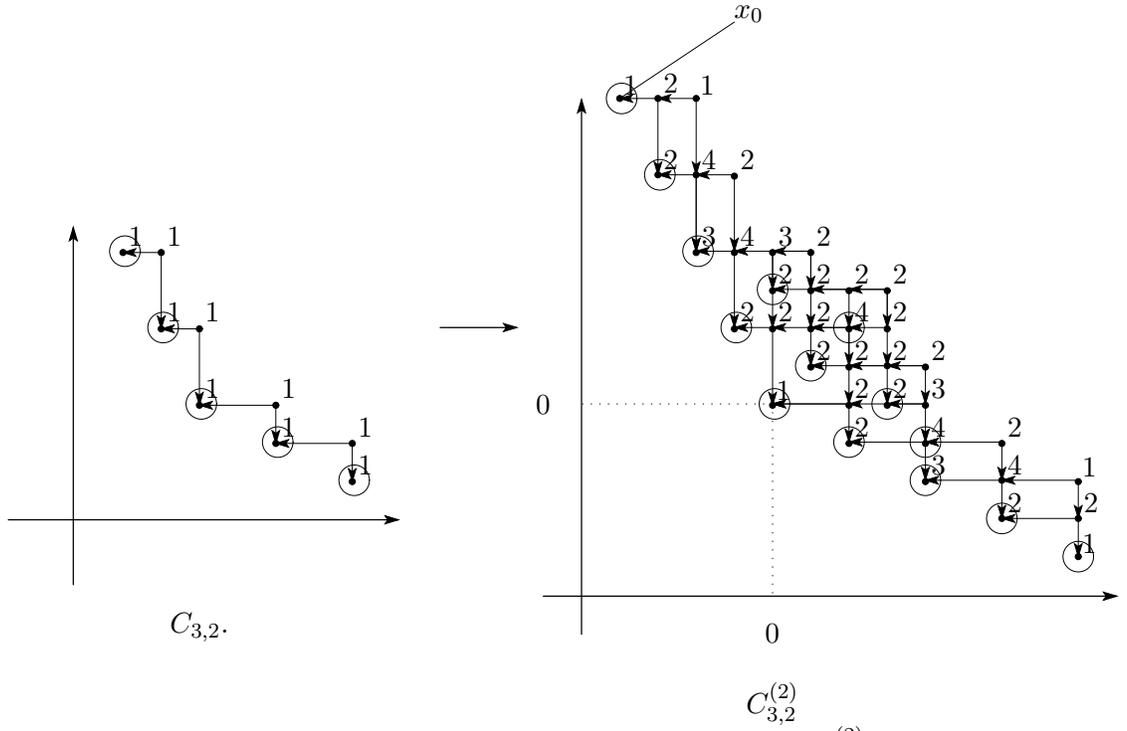}\caption{The generators and differentials of $C_{3,2}$ and $C_{3,2}^{(2)}$ and the element $x_0$.}\label{C32}\end{center}\end{figure}
\qed
\begin{rmk}
This computation of the values of $V$ can be also done in the cases of more general torus knots,
however to compute the such chain complexes derails our aims, and we do not write down it here.
It, however, seems that the determination of the rational 4-ball bound-ness for the double branched covers of twisted Whitehead double of any torus knot or more general knots becomes a challenging work.
\end{rmk}

As a corollary we give a sufficient condition to satisfy $\delta(D_+(K,n))=0$ for a knot $K$ with non-negative $\tau(K)$.
\begin{cor}
\label{deltazerocor}
Let $K$ be a knot in $S^3$ with $\tau(K)\ge 0$.
If $n\ge 2\tau(K)$, then 
$$\delta(D_+(K,n))=0.$$
\end{cor}
\begin{proof}
We claim that if $k=2\tau(K)=\tau(K\#K^r)$, then we have $V_k=0$.
Then, by the decreasing property $V_k\ge V_{k+1}\ge 0$, the assertion required holds.

Let $C$ denote $CFK^\infty(K\#K^r)$ and let $k$ denote $2\tau(K)$.
There exist a generator $x\in C\{(0,k)\}$ and some element $\alpha\in C\{\max\{i,j- k\}\ge 0\}$ such that
a non-zero class $[x+\alpha]\in H_\ast(A_k^+)$, and its image by $\frak{v}_k^+:H_\ast(A_k^+)\to H_\ast(B^+)={\mathcal T}^+$ is the bottom generator.
Thus this means that $[x+\alpha]\neq 0$.
Clearly, $U\cdot [x+\alpha]\neq 0$, $[x+\alpha]$ is the bottom generator in $A^T_k$.
This means $V_k=0$.
See {\sc Figure}~\ref{Aplusk}.

Therefore, for any $n\ge 2\tau(K)$, we have
$$\delta(D_+(K,n))=2d(S^3_\frac{4n+1}{2}(K\#K^r),i_0)=2d(L(4n+1,2),i_0)-4V_{n}=0-0=0.$$
\qed\end{proof}
\begin{figure}[htbp]\begin{center}\input{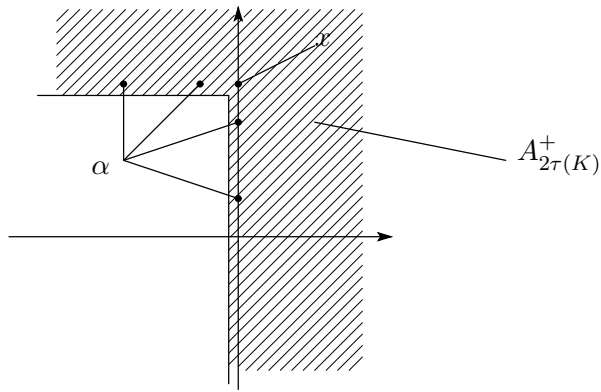}\caption{The generator $x$ and some element $\alpha$ in $A_{2\tau(K)}^+$.}\label{Aplusk}\end{center}\end{figure}
If $0\le n\le 2\tau(K)$, then the behavior of $V_n$  depends on the filtered chain complex with respect to $K$.

{\bf Proof of Theorem~\ref{lspathm}.}
Let $K$ be a positive L-space knot.
Then the genus $g(K)$ coincides with $\tau(K)$.
We choose the minimal chain complex $CFK^\infty$, which the dimension of $CFK^\infty\{(i,j)\}$ is at most one.
Since $CFK^\infty\{(0,\tau(K)-1)\}$ has dimension one due to \cite{H2, MT2}, the chain complex $CFK^\infty(K)$ is as in {\sc Figure}~\ref{lsp}.

As described in {\sc Figure}~\ref{lsp}, $V_{\tau(K)}=0$ and $V_{\tau(K)-1}=1$ hold.
Thus, from Corollary~\ref{deltazerocor}, if $n\ge 2\tau(K)$ holds, then $\delta(D_+(K,n))=0$.
From {\sc Figure}~\ref{lsp} we have
$$\delta(D_+(K,2\tau(K)-1))=-4V_{2\tau(K)-1)}=-4.$$
This means $t_\delta(K)=2\tau(K)$.
\begin{figure}[htbp]
\begin{center}\input{lsp}\end{center}\caption{The minimal $CFK^\infty(K)$ of an L-space knot $K$ (the left picture).
 $A_{2\tau(K)-1}^+$, $A_{2\tau(K)}^+$ for $K\#K^r$ (the central and right pictures). The points indicated by the circles represent the unique generator for each of homologies.}\label{lsp}\end{figure}
\hfill\qed

\subsection{An obstruction by Owens and Strle.}
If $K$ is a slice knot, then the double branched cover $\Sigma_2(K)$ must bound a rational 4-ball.
To show Theorem~\ref{RationalBall} (the rational 4-ball bound-ness), we use the following refinement of the $\delta$-invariant by Owens and Strle.
\begin{prop}[\cite{Ow}]
\label{OW}
Let $Y$ be a rational homology 3-sphere bounding a rational 4-ball $X$.
If the order of $H^2(Y)$ is $h=t^2$, then 
$$d(Y,\frak{t}_0+\beta)=0$$
for any $\beta\in {\mathcal T}\subset H^2(Y)$,
where $t_0$ is a Spin$^c$ structure, ${\mathcal T}$ is the image of the map $H^2(X)\to H^2(Y)$,
and $|{\mathcal T}|=t$.
\end{prop}

By using Proposition~\ref{4n+12KK} and \ref{OW} for the half-integer surgery of $T_{2,2p+1}\#T_{2,2p+1}^r$ we prove Theorem~\ref{RationalBall}.

{\bf Proof of Theorem~\ref{RationalBall}.}
Suppose that 
$$X_{s,p,m}:=\Sigma_2(D_+(T_{s.sp+1},m(m+1)))$$
bounds a rational 4-ball.

Since the canonical Spin$^c$ structure corresponds to $i_0=2m(m+1)+1$, 
Owens and Strle's subset $\frak{t}_0+{\mathcal T}$ is $\{i_0+\ell(2m+1)|0\le|\ell|\le m\}$,
because the subgroup of order $2m+1$ in $H^2(X_{s,p,m},{\Bbb Z})\cong {\Bbb Z}/(2m+1)^2{\Bbb Z}$
is unique.

By using the formula~(\ref{L2formula}), we have	
$$d(L((2m+1)^2,2),i_0+\ell(2m+1))=\begin{cases}2\ell_1(\ell_1-1)&\ell=2\ell_1-1\\2\ell_1^2&\ell=2\ell_1,\end{cases}$$
$$V_{s,p,\lfloor\frac{i_0+\ell(2m+1)}{2}\rfloor}=V_{s,p,m(m+1+\ell)+\lfloor\frac{\ell+1}2\rfloor}$$
$$H_{s,p,\lfloor\frac{i_0+\ell(2m+1)-(2m+1)^2)}{2}\rfloor}=V_{s,p,m(m+1-\ell)-\lfloor\frac{\ell}{2}\rfloor}$$
and 
$$\max\{V_{s,p,\lfloor\frac{i_0+\ell(2m+1)}{2}\rfloor},H_{s,p,\lfloor\frac{i_0+\ell(2m+1)-(2m+1)^2)}{2}\rfloor}\}=V_{s,p,m(m+1)-m\ell-\lfloor\frac{\ell}{2}\rfloor}.$$
Thus we have
$$d(X_{s,p,m},i_0+\ell(2m+1))=\begin{cases}2\ell_1(\ell_1-1)-2V_{s,p,m(m+1-\ell)-\ell_1+1}&\ell=2\ell_1-1\\2\ell_1^2-2V_{s,p,m(m+1-\ell)-\ell_1}&\ell=2\ell_1.\end{cases}$$

Suppose that $\ell=2\le m$ holds.
Then we have $V_{s,p,m^2-m-1}=1$.
In the case of $s=2$, since $m^2-m-1$ is odd, we obtain $m^2-m-1=2p-1$.
In the case of $s=3$, since $m^2-m-1\equiv \pm1\bmod 6$, we have $m^2-m-1=6p-1$.
If $m=2$, then $(s,p)=(2,1)$ holds. 

Suppose that $m\ge 3$ and $\ell=3\le m$.
Then we have $V_{s,p,m^2-2m-1}=2$.
In the case of $s=2$, we have $m^2-2m-1=2s-3$, or $2s-4$, thus $(m,p)=(3,3)$.
In the case of $s=3$, since $m^2-2m-1\neq 3,6\bmod 6$, we have $m^2-2m-1=6p-5,6p-4$ holds.
Thus $(p,m)=(2,4),(1,3)$.

Therefore, we obtain $(s,p,m)=(2,1,2)$, $(2,3,3)$, $(3,1,3)$, and $(3,2,4)$.\hfill\qed

The obstruction (Theorem~\ref{RationalBall}) by Owens and Strle \cite{Ow} for $K=T_{2,2p+1}$ and $T_{3,3p+1}$ coincides with
the condition for $\Sigma_2(D_+(K,n))$ to bound a rational 4-ball.
It would be unlikely that the coincidence can be extended to other torus knots version.
For example, the double branched covers of $D_+(T_{3,5},12)$, $D_+(T_{3,11},30)$ and $D_+(T_{3,20},56)$ satisfy the condition by Owens and Strle,
however we do not know whether rational 4-ball bounds of the manifolds exist or not.
\begin{que}
Let $K'$ be $D_+(T_{3,5},12)$, $D_+(T_{3,11},30)$ or $D_+(T_{3,20},56)$.
Then does the double branched covers of $K'$ bound a rational 4-ball?
\end{que}

{\bf Proof of Theorem~\ref{2712}.}
The knots $K'=D_+(T_{2,7},12)$ and $D_+(T_{3,7},20)$ are not slice by Collins' result \cite{JC}.
We prove that each of the double branched cover of $K'$ actually bounds a rational 4-ball.
Each of the double branched covers is $\Sigma_2(K')=S^3_{49/2}(T_{2,7}\#T_{2,7}^r)$, $S^3_{81/2}(T_{3,7}\#T_{3,7}^r)$.
They are diffeomorphic to each of the last diagrams in {\sc Figure}~\ref{rationalh} and ~\ref{rationalhd}.
The $0$-framed components in the last diagrams consist of slice knots $S_1$ and $S_2$ respectively.
Indeed, the knot $S_1$ (in the $K'=D_+(T_{2,7},12)$ case) is $10_{155}$ in the Rolfsen table.
It is a well-known slice knot.
The sliceness of $S_2$ (in the $K'=D_+(T_{3,7},20)$ case) is proven in {\sc Figure}~\ref{ribbon3}.
\begin{figure}[bthp]
\begin{center}\includegraphics{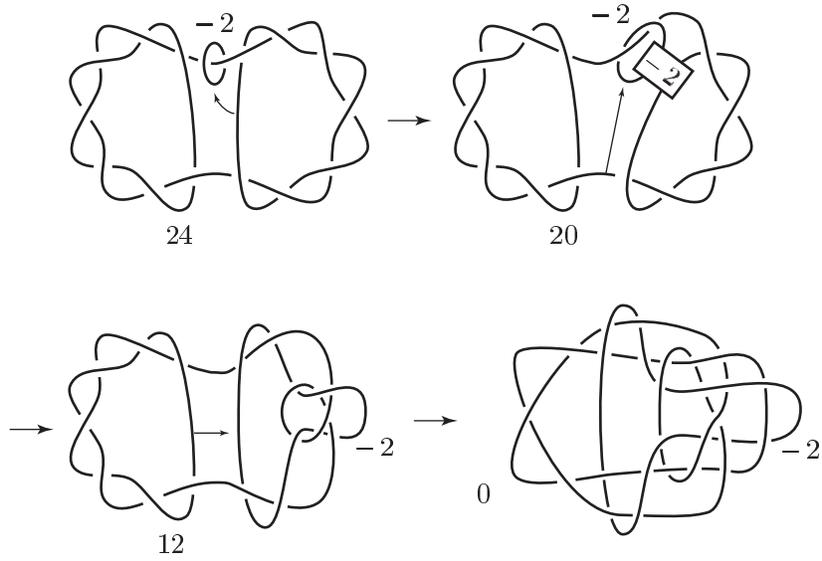}\caption{A diffeomorphism of $S^3_{49/2}(T_{2,7}\#T_{2,7}^r)$.}\label{rationalh}\end{center}
\end{figure}
\begin{figure}[bthp]
\begin{center}\includegraphics{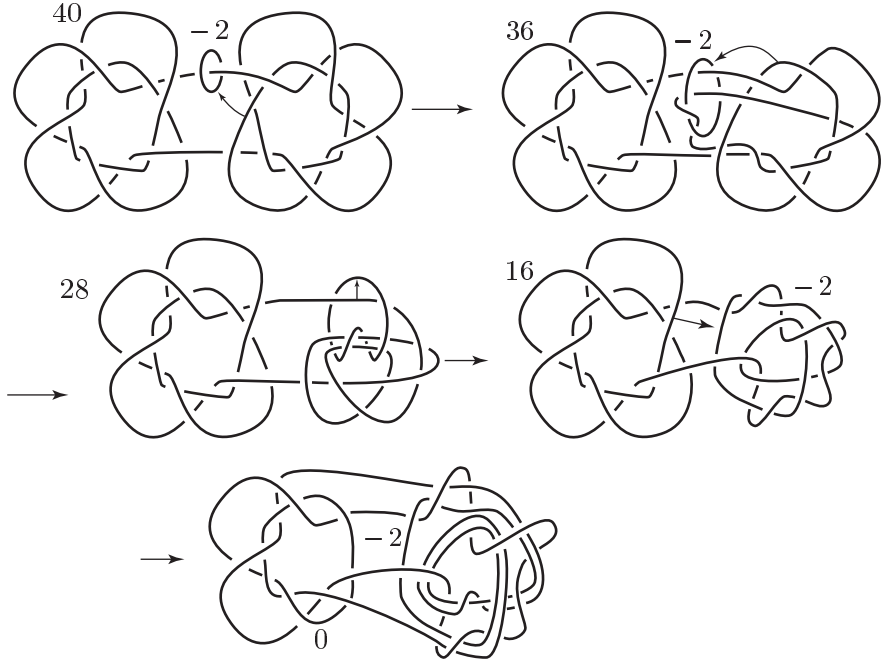}\caption{A diffeomorphism of $S^3_{81/2}(T_{3,7}\#T_{3,7}^r)$.}\label{rationalhd}\end{center}
\end{figure}
\begin{figure}[btbp]
\begin{center}\includegraphics{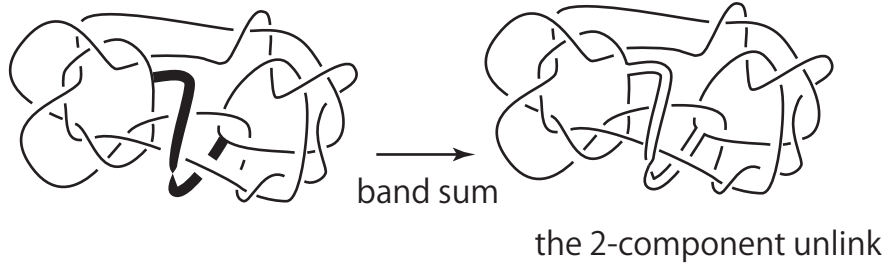}\caption{The sliceness of $S_2$.}\label{ribbon3}\end{center}
\end{figure}
Let $B_7$, and $B_9$ be the slice disk complements of $S_1$ and $S_2$ together with the $-2$-framed 2-handles.
See {\sc Figure}~\ref{rationalh2} and \ref{ratb}.
By computing $H_1$ and $H_2$ of $B_7$ and $B_9$ from the last pictures in {\sc Figure}~\ref{rationalh2} and \ref{ratb} respectively, we obtain immediately
$$H_1(B_7)\cong {\Bbb Z}/7{\Bbb Z}\text{ and }H_2(B_7)\cong 0$$
and 
$$H_1(B_9)\cong {\Bbb Z}/9{\Bbb Z}\text{ and }H_2(B_9)\cong 0.$$
Hence, $B_7$ and $B_9$ are rational 4-balls bounding the double branched covers of $D_+(T_{2,7},12)$ and $D_+(T_{3,7},20)$ respectively.
\begin{figure}[btbp]
\begin{center}\includegraphics{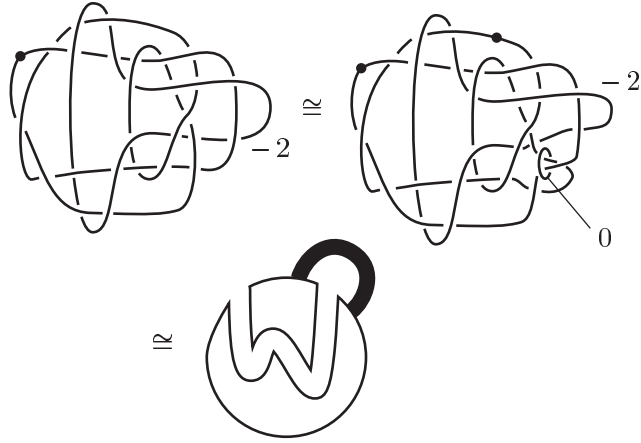}\caption{A rational 4-ball $B_7$: A slice disk complement of $S_1=10_{155}$ attaching $-2$-framed 2-handle.}\label{rationalh2}\end{center}
\end{figure}
\begin{figure}[btbp]
\begin{center}\includegraphics{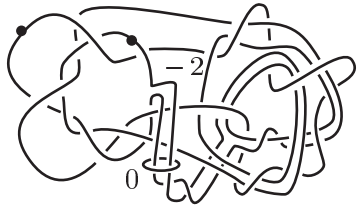}\caption{A rational 4-ball $B_9$: A slice disk complement of $S_2$ attaching $-2$-framed 2-handle.}\label{ratb}\end{center}
\end{figure}
\qed

\end{document}